\documentclass[a4paper,12pt]{amsart}
\usepackage{fullpage}
\usepackage[latin2]{inputenc}
\usepackage[english]{babel} %[magyar,english]

\usepackage{tikz}
\usepackage{amsmath}
\usepackage{amsthm}
\usepackage{amssymb}
\usepackage{cite}
\usepackage{amsrefs}
\usepackage{rotating}
\usepackage{graphicx}
\usepackage{subfigure}

\newtheorem{thm}{Theorem}[section]
\newtheorem{prop}[thm]{Proposition}

\newtheorem{conj}[thm]{Conjecture}

\newtheorem{lem}[thm]{Lemma}
\newtheorem{defi}[thm]{Definition}

\newtheorem{claim}[thm]{Claim}

\theoremstyle{remark}
\newtheorem*{remark}{Remark}

\begin{document}

\title{Triangle-different Hamiltonian paths}

\makeatother
\author{István Kovács}
\address{Department of Control Engineering and Information Technology \\ Budapest University of Technology and Economics}
\email[István Kovács]{kovika91@gmail.com}
\thanks{\noindent The research of the first author was
supported by National Research, Development and Innovation Office NKFIH, K-111827.}

\author{Daniel Soltész}
\address{Department of Computer Science and Information Theory \\ Budapest University of Technology and Economics}
\email[Daniel Soltész]{solteszd@math.bme.hu}
\thanks{\noindent The research of the second author was
supported by the Hungarian Foundation for Scientific Research Grant (OTKA) No. 108947}

\begin{abstract}

Let $G$ be a fixed graph. Two paths of length $n-1$ on $n$ vertices (Hamiltonian paths) are $G$-different if there is a subgraph isomorphic to $G$ in their union. In this paper we prove that the maximal number of pairwise triangle-different Hamiltonian paths is equal to the number of balanced bipartitions of the ground set, answering a question of Körner, Messuti and Simonyi.

%We also determine the maximal number of Hamiltonian cycle-different Hamiltonian paths when the size of the ground set is a prime.

\end{abstract}

%\keywords{{\em Keywords}: Hamiltonian path, triangle, union.}

\maketitle

\section{Introduction}

%Hamiltonian paths of the complete graph on $n$ vertices are in obvious relationship with permutations of the set $[n]:=\{1,\dots,n\}$.
%Investigations of the size of largest sets of permutations pairwise
%satisfying a prescribed relation was initiated in \cite{colliding} in
%relation with the capacity of infinite graphs, a notion analogous to Shannon's graph capacity concept.
%Investigations of the size of largest sets of permutations pairwise satisfying a prescribed relation was initiated in \cite{colliding} in relation with the capacity of infinite graphs, a notion analoguous to Shannon's graph capacity concept.
Problems concerning the size of largest sets of permutations pairwise
satisfying a prescribed relation has a large literature, see e.g.
\cites{DF, EFP}. Investigations of a special type of such problems
related to the Shannon capacity of infinite graphs, a notion analogous
to Shannon's graph capacity concept, was initiated in
\cite{colliding}.
Two permutations $\pi_1,\pi_2$ of $[n]:=\{1,\dots,n\}$ are called colliding in \cite{colliding} if there is an index $i\in [n]$ such that $|\pi_1(i)-\pi_2(i)|= 1$.

\begin{conj}[Körner, Malvenuto]\cite{colliding}\label{colliding}
The maximal number of pairwise colliding permutations of $[n]$ is $\binom{n}{\left \lfloor \frac{n}{2} \right \rfloor}$.
\end{conj}

In Conjecture~\ref{colliding} $\binom{n}{\left \lfloor \frac{n}{2} \right \rfloor}$ is best possible, since permutations containing numbers of the same parity at every position do not collide, so the maximal number of pairwise colliding permutations is at most the number of "parity patterns", that is, the number of ways $\lceil\frac{n}{2}\rceil$ odd and $\lfloor{\frac{n}{2}}\rfloor$ even numbers can be placed on $n$ positions if only the parity of the numbers matter, their actual value do not. The largest construction known contains roughly $1.81^n$ permutations (see \cite{infinite}). Conjecture~\ref{colliding} triggered the investigation of several problems of the same flavor that concern the maximal number of permutations any pair of which satisfy some specified constraint, see \cites{KMS, KSS, infinite}.  There is a natural relationship between Hamiltonian paths and permutations. In this paper we focus on problems of the above type that can naturally be formulated in terms of Hamiltonian paths.

The union of two graphs $H_1$ and $H_2$ on the same vertex set is the graph on this common vertex set having $E(H_1) \cup E(H_2)$ as edge set. Let $G$ be some fixed graph. We say that two Hamiltonian paths are $G$-different if their union contains $G$ as a subgraph.
%A family of Hamiltonian paths is $G$-different if every pair of paths from this family is $G$-different.
The maximal number of pairwise $G$-different Hamiltonian paths has been studied for various $G$ in \cites{komesi, locsep, k=4}. The problem is uninteresting when $G$ is contained in a Hamiltonian path.
%When $G$ is a path the problem is trivial as every union of two different Hamiltonian paths contains $G$ as a subgraph.
(The maximal size of $G$-different families in these cases is simply $\frac{n!}{2}$.) A somewhat more interesting case is that of $K_{1,3}$-different Hamiltonian paths. It is easy to see that the union of two Hamiltonian paths does not contain a vertex of degree $3$ if and only if the union itself is a Hamiltonian cycle. Thus the maximal size of $K_{1,3}$-different Hamiltonian paths is $\frac{(n-1)!}{2}$. The first few choices for $G$, where the problem becomes more difficult, are: $K_3$, $K_4$, $C_4$, $K_{1,4}$. Until now even the correct order of magnitude was unknown for these families, for the best lower and upper bounds known so far see Table~\ref{results}. In this paper we determine the maximal number of pairwise triangle-different Hamiltonian paths on $n$ vertices exactly.

%
%\begin{table}[htbp]
%\begin{center}
%\renewcommand*{\arraystretch}{1.5}
%  \begin{tabular}{| c | c |c|c|c|}
%    \hline
%   \text{rough lower bound}  &lower bound & G & upper bound &\text{rough upper bound}  \\ \hline
%    $1.18^n$  &&$K_4$ &&  $ 1.5^n$  \\ \hline
%     $1.58^n$  && $K_3$ &&  $2^n$  \\ \hline
%     $n^{\frac{1}{2}n}$ &&$  C_4 $&&$ n^{\frac{3}{4}n} $ \\ \hline
%   $ n^{\frac{n}{2}}2^{-0.75n} $ &$ \frac{\left( \left\lfloor \frac{n}{2} \right\rfloor -1 \right)!}{2^{\left \lfloor \frac{n}{4} \right \rfloor}}$ &$ K_{1,4} $&$ \frac{n!}{\left\lfloor \frac{n}{2} \right\rfloor!  \, 2^{\left\lfloor \frac{n}{2} \right\rfloor}}$ &$ n^{\frac{n}{2}}2^{-0.72n}$ \\ \hline
%  \end{tabular}
%\end{center}
%\caption{The rough order of magnitude of lower and upper bounds for the maximal size of $G$ different Hamiltonian paths. For the precise bounds see \cite{komesi,k=4,locsep}.}
%\label{fi}
%\end{table}

%The union of two graphs on the same vertex set is the graph having the edge set that is the union of the edge sets of the two graphs.

%In a series of papers \cites{connector,locsep,degree4,k=4,komesi,indep}, questions of the following type are studied. What is the maximal number of graphs from a given graph class, such that the union of every two graphs satisfies a certain condition.
 Körner, Messuti and Simonyi made the following observation.

\begin{prop} \cite{komesi} \label{prop}
The maximal number of Hamiltonian paths such that every pairwise union contains an odd cycle is equal to the number of balanced bipartitions of the vertex set.  That is, on $2n+1$ vertices this number is $\binom{2n+1}{n}$ and on $2n+2$ vertices it is $\frac{1}{2}\binom{2n+2}{n+1}=\binom{2n+1}{n}$.
\end{prop}

The upper bound follows by observing that a Hamiltonian path is a bipartite graph with a balanced bipartition and the union of two paths with the same bipartition is a bipartite graph which clearly cannot contain any odd cycle. On the other hand, if for every balanced bipartition we choose an arbitrary Hamiltonian path that corresponds to it, then we obtain a good family.

The authors in \cite{komesi} asked whether the same upper bound can be attained with pairwise triangle-different Hamiltonian paths.
They verified that this is possible up to $5$ vertices. By an easy product construction this yields a lower bound of roughly $1.58^n$.

An affirmative answer to the above question may be interpreted by saying that insisting on a triangle in the pairwise unions is not more restrictive than requiring just any odd cycle. As noted in \cite{komesi}, there are some famous theorems of this kind that show the indifference of specifying the triangle among odd cycles in certain situations.
%Their question was further motivated by the existence of several famous results the content of which can be formulated by saying that "the answer is the same whether we ask for an odd cycle or a triangle".
For example, the maximum number of edges in an odd-cycle-free (bipartite) graph is $\left \lfloor \frac{n^2}{4} \right \rfloor$ and the Mantel-Turán theorem states that this number is the same for triangle-free graphs. Another slightly related example is the following celebrated theorem, where the authors are interested in the intersection of general graphs instead of unions of Hamiltonian paths.

 \begin{thm}[D. Ellis, Y. Filmus, E. Friedgut \cite{trintersect}] \label{intersect}

The following three numbers are equal.
\begin{itemize}
\item The maximum number of $n$-vertex graphs such that every pairwise intersection contains an odd cycle.
\item The maximum number of $n$-vertex graphs such that every pairwise intersection contains a triangle.
\item The number of $n$-vertex graphs that contain a fixed triangle.
\end{itemize}
\end{thm}

%Theorem~\ref{intersect} deals with a similar question where instead of unions, we are interested in intersections without any restrictions on the graphs.

%The reason that we restrict our attention on Hamiltonian paths is that they correspond to permutations  and the problems that ultimately led the authors of \cite{komesi} to this question involved permutations, see \cite{colliding}.

\smallskip

We will show that the union version of Theorem \ref{intersect} is much easier in Section \ref{other graphs}. The main result of the present paper is the following theorem which answers the question of Körner, Messuti and Simonyi affirmatively.

\begin{thm} \label{main}
The maximum number of Hamiltonian paths such that every pairwise union contains a triangle is equal to the number of balanced bipartitions of the ground set.
\end{thm}

Since the same upper bound holds as in Proposition~\ref{prop}, the main challenge is to construct a family of this size that satisfies the condition.

We prove Theorem~\ref{main} in Section~\ref{proof}. In Section~\ref{Hcdiff} we also investigate the case one can consider the "other extreme", where we look for large families of Hamiltonian paths any two of which forming a union that contains a Hamiltonian cycle.
%we direct our attention to the other extreme case,
We prove that the maximal number of such Hamiltonian-cycle-different Hamiltonian paths on $n$ vertices is at most $\binom{n}{2}$ and provide a construction which achieves this bound whenever $n$ is prime.
Section~\ref{opr} contains some open problems and concluding remarks.

\begin{table}[htbp]
\begin{center}
\renewcommand*{\arraystretch}{1.8}
  \begin{tabular}{l c c c}
    \hline
  Name &lower bound  & G  & upper bound   \\ \hline
 Körner, Messuti, Simonyi \cite{komesi} &  $1.18^{n-o(n)}$  &$K_4$ &  $ 1.5^{n+o(n)}$  \\
 Kovács, Soltész (this result) &  $2^{n-o(n)}$  & $K_3$ &  $2^{n-o(n)}$ \\
  Cohen, Fachini, Körner \cite{connector} &  $n^{\frac{1}{2}n-o(n)}$ &$  C_4 $&$ n^{\frac{3}{4}n+o(n)}$  \\
 Körner, Monti \cite{locsep} & $ n^{\frac{n}{2}}2^{-1.47n -o(n)} $ & \, $K_{1,4} $ \, & $ n^{\frac{n}{2}}2^{-0.72n +o(n)}$  \\ \hline
  \end{tabular}
\end{center}
\caption{The order of magnitude of the lower and upper bounds for the maximal size of pairwise $G$-different Hamiltonian paths for the first few non-trivial choices for $G$.}
\label{results}
\end{table}

\section{Proof of the main theorem} \label{proof}

In this section we prove Theorem~\ref{main}. The number of balanced bipartitions of $2n+1$ and $2n+2$ vertices is the same, namely $\binom{2n+1}{n}$. Thus it is enough to construct $\binom{2n+1}{n}$ triangle-different Hamiltonian paths on $2n+1$ vertices, since we can add an additional vertex and extend every original Hamiltonian path by an edge when we want a good construction of the same size on $2n+2$ vertices. We say that two Hamiltonian paths are \textit{compatible} if their union contains a triangle. We also say that a set of Hamiltonian paths is compatible if each pair of them is compatible. We will construct a compatible set of Hamiltonian paths of the required size. Within our construction the Hamiltonian paths will belong to several groups (called types) according to the following definition.

\begin{defi}\label{defi:Gtype}
Let $G$ be a graph with weighted edges where the weights are $1$ or $2$. We say that a Hamiltonian path $H$ on the vertex set of $G$ is $G$-type if the following holds. For every $u$ and $v$ that are connected in $G$ by an edge of weight $w$, the distance of $u$ and $v$ in $H$ is exactly $w$.
\end{defi}

By weighted graph we will always mean a weighted graph where the edges get weight $1$ or $2$. When drawing a weighted graph we will draw the edges of weight $1$ as ordinary lines, and the edges of weight $2$ by dashed lines. See Figure~\ref{example} for an illustration of Definition~\ref{defi:Gtype}.

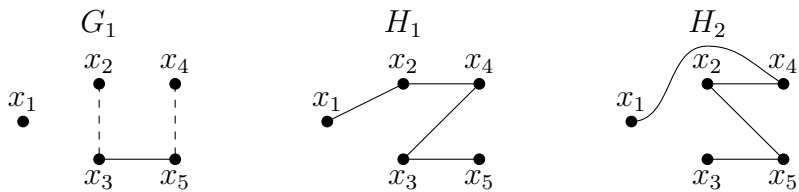
\begin{figure}[htbp,scale=0.5]
\begin{center}
\begin{tikzpicture}
\filldraw[black] (0,0) circle (2pt)node[anchor=south] {$x_1$};
\filldraw[black] (1,0.5) circle (2pt)node[anchor=south] {$x_2$};
\filldraw[black] (1,-0.5) circle (2pt)node[anchor=north] {$x_3$};
\filldraw[black] (2,0.5) circle (2pt)node[anchor=south] {$x_4$};
\filldraw[black] (2,-0.5) circle (2pt)node[anchor=north] {$x_5$};
\draw (1,-0.5)--(2,-0.5);
\draw[dashed] (1,0.5)--(1,-0.5);
\draw[dashed] (2,0.5)--(2,-0.5);
\draw (1,1.3) node[]{$G_1$};

\filldraw[black] (4,0) circle (2pt)node[anchor=south] {$x_1$};
\filldraw[black] (5,0.5) circle (2pt)node[anchor=south] {$x_2$};
\filldraw[black] (5,-0.5) circle (2pt)node[anchor=north] {$x_3$};
\filldraw[black] (6,0.5) circle (2pt)node[anchor=south] {$x_4$};
\filldraw[black] (6,-0.5) circle (2pt)node[anchor=north] {$x_5$};
\draw  (4,0) --(5,0.5)  ;
\draw  (5,0.5) --(6,0.5)  ;
\draw  (6,0.5) -- (5,-0.5) ;
\draw  (5,-0.5) -- (6,-0.5) ;
\draw (5,1.3) node[]{$H_1$};

\filldraw[black] (8,0) circle (2pt)node[anchor=south] {$x_1$};
\filldraw[black] (9,0.5) circle (2pt)node[anchor=south] {$x_2$};
\filldraw[black] (9,-0.5) circle (2pt)node[anchor=north] {$x_3$};
\filldraw[black] (10,0.5) circle (2pt)node[anchor=south] {$x_4$};
\filldraw[black] (10,-0.5) circle (2pt)node[anchor=north] {$x_5$};
\draw (8,0) to[out=0,in=180]  (9,1)  ; %(9,1) is only used to control the curvature of this edge
\draw (9,1) to[out=0,in=150]  (10,0.5)  ;
\draw (10,0.5) --  (9,0.5)  ;
\draw (9,0.5) -- (10,-0.5)   ;
\draw (10,-0.5) -- (9,-0.5)   ;
\draw (9,1.3) node[]{$H_2$};
\end{tikzpicture}
\caption{Both $H_1$ and $H_2$ are $G_1$-type.}
\label{example}
\end{center}
\end{figure}

Observe that if $G_1$ and $G_2$ are weighted graphs that share an edge that has a different weight in $G_1$ and in $G_2$, then every $G_1$-type Hamiltonian path is compatible with every $G_2$-type Hamiltonian path. Hence we define compatibility for weighted graphs too.

\begin{defi}
Two weighted graphs $G_1$ and $G_2$ are compatible if they share an edge that has different weight in $G_1$ and $G_2$. We say that a family of weighted graphs is compatible if every pair of weighted graphs from the family is compatible.
\end{defi}

Our strategy is to first build a compatible family of weighted graphs %$\{G_1,G_2, \ldots , G_i , \ldots\}$,
on a ground set of odd size. Then for each weighted graph $G_i$ in the family construct a set of compatible Hamiltonian paths that are $G_i$-type. To obtain what we want, we need that we build a suitable family so that we can construct enough Hamiltonian paths from it.

%Let us start to define the special class of weighted graphs which we will be working with.

\begin{defi}
A $k$-ladder is a weighted graph on the vertices $\{v_1, \ldots, v_k, w_1, \ldots, w_k\}$ where the edges of weight $1$ are $\{(v_1,w_1),\ldots, (v_k,w_k)\}$ and the edges of weight $2$ are $\{(v_1,v_2),(v_2,v_3), \ldots, (v_{k-1},v_k),(w_1,w_2),(w_2,w_3), \ldots, (w_{k-1},w_k)\}$. We say that a weighted graph is a (weighted) ladder if it is a $k$-ladder for some $k$, see Figure~\ref{ladders}. We also call the vertices $(v_1,w_1)$ and $(v_k,w_k)$ the top and the bottom of the ladder respectively, and we assume that for each ladder it is fixed which one of its edges is the top and which is the bottom.
\end{defi}

We remark that we do not introduce any notation for distinguishing the top and bottom of ladders except for drawing the ladders this way on the figures.

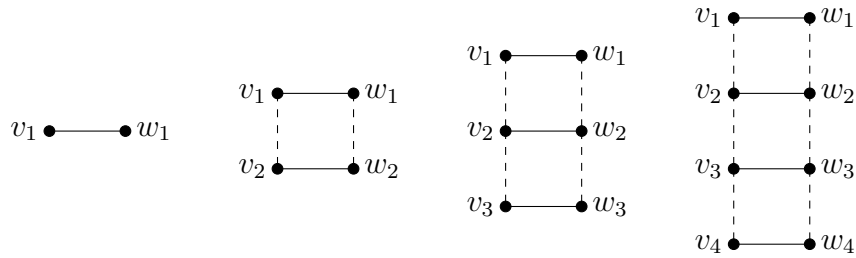
\begin{figure}[htbp,scale=0.5]
\begin{center}
\begin{tikzpicture}
\filldraw[black] (-7,-0.5) circle (2pt)node[anchor=east] {$v_1$};
\filldraw[black] (-6,-0.5) circle (2pt)node[anchor=west] {$w_1$};
\draw (-7,-0.5)--(-6,-0.5);

\filldraw[black] (-4,0) circle (2pt)node[anchor=east] {$v_1$};
\filldraw[black] (-3,0) circle (2pt)node[anchor=west] {$w_1$};
\filldraw[black] (-4,-1) circle (2pt)node[anchor=east] {$v_2$};
\filldraw[black] (-3,-1) circle (2pt)node[anchor=west] {$w_2$};
\draw (-4,0)--(-3,0);
\draw (-4,-1)-- (-3,-1);
\draw[dashed](-4,0) -- (-4,-1);
\draw[dashed](-3,0) -- (-3,-1);

\filldraw[black] (-1,0.5) circle (2pt)node[anchor=east] {$v_1$};
\filldraw[black] (-1,-0.5) circle (2pt)node[anchor=east] {$v_2$};
\filldraw[black] (-1,-1.5) circle (2pt)node[anchor=east] {$v_3$};

\filldraw[black] (0,0.5) circle (2pt)node[anchor=west] {$w_1$};
\filldraw[black] (0,-0.5) circle (2pt)node[anchor=west] {$w_2$};
\filldraw[black] (0,-1.5) circle (2pt)node[anchor=west] {$w_3$};
\draw (-1,0.5)--(0,0.5);%steps
\draw (-1,-0.5)--(0,-0.5);
\draw (-1,-1.5)--(0,-1.5);%stepsend
\draw[dashed](-1,0.5) -- (-1,-0.5);
\draw[dashed](-1,-0.5) -- (-1,-1.5);
\draw[dashed](0,0.5) -- (0,-0.5);
\draw[dashed](0,-0.5) -- (0,-1.5);

\filldraw[black] (2,1) circle (2pt)node[anchor=east] {$v_1$};
\filldraw[black] (2,0) circle (2pt)node[anchor=east] {$v_2$};
\filldraw[black] (2,-1) circle (2pt)node[anchor=east] {$v_3$};
\filldraw[black] (2,-2) circle (2pt)node[anchor=east] {$v_4$};

\filldraw[black] (3,1) circle (2pt)node[anchor=west] {$w_1$};
\filldraw[black] (3,0) circle (2pt)node[anchor=west] {$w_2$};
\filldraw[black] (3,-1) circle (2pt)node[anchor=west] {$w_3$};
\filldraw[black] (3,-2) circle (2pt)node[anchor=west] {$w_4$};
\draw (2,1) -- (3,1) ;%steps
\draw (2,0) -- (3,0) ;
\draw (2,-1) -- (3,-1) ;
\draw (2,-2) -- (3,-2) ;%stepsend

\draw[dashed](2,1) -- (2,0);
\draw[dashed](2,0) -- (2,-1);
\draw[dashed](2,-1) -- (2,-2);
\draw[dashed](3,1) -- (3,0);
\draw[dashed](3,0) -- (3,-1);
\draw[dashed](3,-1) -- (3,-2);
\end{tikzpicture}
\caption{$k$-ladders for $k=0,1,2,3$. The edge $(v_1,w_1)$ is the top of each ladder.}
\label{ladders}
\end{center}
\end{figure}

Now we define a special class of weighted graphs.

%We will only work with weighted graphs from this class for the rest of the paper.

\begin{defi} We call a weighted graph \textit{properly laddered} if it is the disjoint union of ladders, an isolated vertex called the apex, and a so called residual part. This residual part can either be an empty graph, a single edge or the union of two vertex-disjoint paths on $m+2$ and $m$ vertices, respectively, for some positive integer $m$. All the edges in the residual part get weight~$2$.

%Let $\mathcal{G}$ be the class of weighted graphs that consists of the disjoint union of a single isolated vertex, some weighted ladders and possibly at most two paths, such that the number of vertices in the paths differ by exactly two and the weights of the edges of the paths are $2$, see Figure~\ref{structure}. We call these paths the residual part of the weighted graph, and we call the isolated vertex the apex. We allow the residual part to consist of only a single edge of weight two, as it is the union of a path with $2$ vertices and a path with $0$ vertices.
\end{defi}

\begin{figure}[htbp,scale=0.5]
\begin{center}
\begin{tikzpicture}
\filldraw[black] (-8.5,-0.5) circle (2pt)node[anchor=east] {$x_1$};

\filldraw[black] (-7,-0.5) circle (2pt)node[anchor=east] {$x_2$};
\filldraw[black] (-6,-0.5) circle (2pt)node[anchor=west] {$x_3$};
\draw (-7,-0.5)--(-6,-0.5);

\filldraw[black] (-4,0) circle (2pt)node[anchor=east] {$x_4$};
\filldraw[black] (-3,0) circle (2pt)node[anchor=west] {$x_5$};
\filldraw[black] (-4,-1) circle (2pt)node[anchor=east] {$x_6$};
\filldraw[black] (-3,-1) circle (2pt)node[anchor=west] {$x_7$};
\draw (-4,0)--(-3,0);
\draw (-4,-1)-- (-3,-1);
\draw[dashed](-4,0) -- (-4,-1);
\draw[dashed](-3,0) -- (-3,-1);

\filldraw[black] (-1,-0.5) circle (2pt)node[anchor=east] {$x_8$};
\filldraw[black] (0,-0.5) circle (2pt)node[anchor=west] {$x_9$};
\draw (-1,-0.5)--(0,-0.5);

\filldraw[black] (2,1) circle (2pt)node[anchor=east] {$x_{10}$};
\filldraw[black] (2,0) circle (2pt)node[anchor=east] {$x_{11}$};
\filldraw[black] (2,-1) circle (2pt)node[anchor=east] {$x_{12}$};
\filldraw[black] (2,-2) circle (2pt)node[anchor=east] {$x_{13}$};
\filldraw[black] (3,1) circle (2pt)node[anchor=west] {$x_{14}$};
\filldraw[black] (3,0) circle (2pt)node[anchor=west] {$x_{15}$};
\draw[dashed](2,1) -- (2,0);
\draw[dashed](2,0) -- (2,-1);
\draw[dashed](2,-1) -- (2,-2);
\draw[dashed](3,1) -- (3,0);
\end{tikzpicture}
\caption{A properly laddered weighted graph. The apex is $x_1$ and the subgraph spanned by the vertices $x_{10},x_{11},x_{12},x_{13},x_{14},x_{15}$ is the residual part. }
\label{structure}
\end{center}
\end{figure}
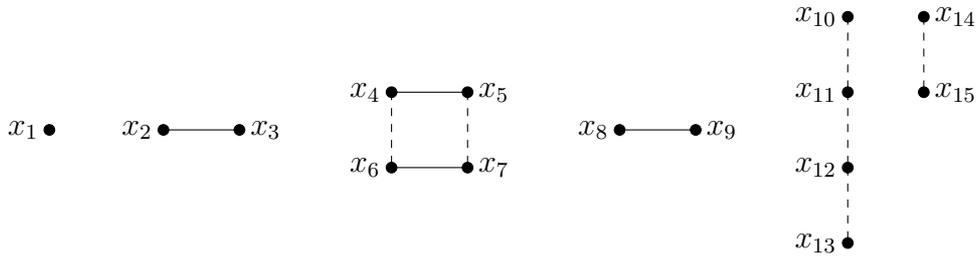

From now on every weighted graph that we use will be properly laddered. 
%We only emphasize it when we actually make use of this property. 
The proof of the next Lemma describes a construction we will use to convert a properly laddered weighted graph to a set of compatible Hamiltonian paths. We will refer to this construction as the {\em Z-swapping construction}.
%Applying it to each member of a family of compatible properly laddered %weighted graphs we will obtain a family of such sets of compatible %Hamiltonian paths the union of which is still a set of compatible %Hamiltonian paths.

%the so called Z-swapping construction which is the only method that we %will use to construct Hamiltonian paths from properly laddered weighted %graphs.

\begin{lem}[Z-swapping construction] \label{swapping}
Let $G$ be a properly laddered weighted graph and let $l$ denote the number of ladders in $G$. There is a set of $2^l$ compatible Hamiltonian paths which are all $G$-type.
\end{lem}
\begin{proof}
We construct a set of $2^l$ compatible $G$-type Hamiltonian paths. 
(For simplicity we may think about our Hamiltonian paths as if they were oriented to make the term "start" of the path more appropriate. We do not really need to deal with directed edges, however.)
 %assume that the Hamiltonian paths are directed, so that we can refer to %the ends of the paths in a convenient way.
Let us denote the apex of $G$ by $x_0$. Each of our Hamiltonian paths starts at the same vertex (that may or may not be $x_0$ according to the rules below). We have three cases depending on the size of the residual part of $G$.

\begin{enumerate}

\item[a)] If the residual part of $G$ is empty, every path starts from the apex $x_0$.
\item[b)] If the residual part of $G$ consists of a single edge $(x_1,x_2)$, then every Hamiltonian path starts with the path $x_1x_0x_2$.
\item[c)] If the residual part of $G$ consists of the two paths $x_1\ldots x_m$ and $y_1\ldots y_{m+2}$ then each Hamiltonian path starts with the path $y_1 x_1 y_2 x_2\ldots y_m x_m y_{m+1}x_0 y_{m+2}$.

\end{enumerate}

See Figure~\ref{residual} for an illustration of the above cases.

% and if $G$ has a residual part, they will share their first few edges too. If $G$ does not contain a residual part, every Hamiltonian path starts from the apex of $G$. If the residual part of $G$ consists of a single edge of weight $2$ every Hamiltonian path starts from a fixed endpoint of this edge, and uses the apex to connect it to the other endpoint with a path of length two. If $G$ has a residual part that consists of two paths: $P_{m+2}$ and $P_{m}$, then let us fix an endpoint of $P_{m+2}$ and every Hamiltonian path starts from this endpoint, see for illustration Figure~\ref{residual} $(a)$. It visits the vertices of the paths alternately until $P_{m}$ is exhausted, see Figure~\ref{residual} $(b)$. Then it uses the apex to get to the end of $P_{m+2}$, see Figure~\ref{residual} $(c)$.

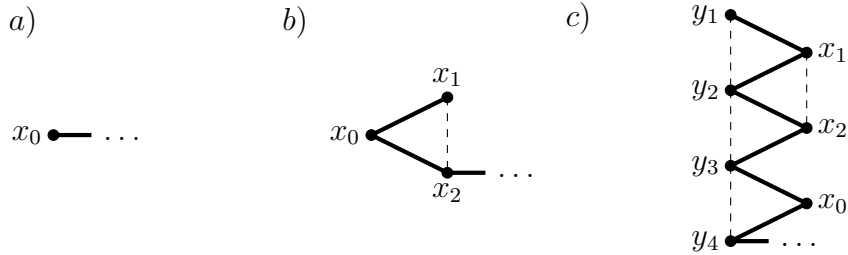
\begin{figure}[htbp,scale=0.5]
\begin{center}
\begin{tikzpicture}
\filldraw[black] (0.4,-0.5) circle (2pt)node[anchor=east] {$x_0$};

\draw[ultra thick] (0.4,-0.5)--(0.9,-0.5);

%
%\filldraw[black] (2,1) circle (3pt)node[anchor=east] {$x_{10}$};
%\filldraw[black] (2,1.2) circle (0pt)node[anchor=south] {$start$};
%\filldraw[black] (2,0) circle (2pt)node[anchor=east] {$x_{11}$};
%\filldraw[black] (2,-1) circle (2pt)node[anchor=east] {$x_{12}$};
%\filldraw[black] (2,-2) circle (2pt)node[anchor=east] {$x_{13}$};
%\filldraw[black] (3,0.5) circle (2pt)node[anchor=west] {$x_{14}$};
%\filldraw[black] (3,-0.50) circle (2pt)node[anchor=west] {$x_{15}$};
%\draw[dashed](2,1) -- (2,0);
%\draw[dashed](2,0) -- (2,-1);
%\draw[dashed](2,-1) -- (2,-2);
%\draw[dashed](3,0.5) -- (3,-0.5);
\draw (0,1) node[]{$a)$};

\draw (1.3,-0.54) node[]{$ \ldots $ };

\draw (3,0) node[]{ };
\draw (0,-2.04) node[]{ };

\end{tikzpicture}
\begin{tikzpicture}
\filldraw[black] (1,-0.5) circle (2pt)node[anchor=east] {$x_0$};

%\filldraw[black] (2,1) circle (2pt)node[anchor=east] {$x_{10}$};
\filldraw[black] (2,0) circle (2pt)node[anchor=south] {$x_{1}$};
\filldraw[black] (2,-1) circle (2pt)node[anchor=north] {$x_{2}$};
%\filldraw[black] (2,-2) circle (2pt)node[anchor=east] {$x_{13}$};
%\filldraw[black] (3,0.5) circle (2pt)node[anchor=west] {$x_{14}$};
%\filldraw[black] (3,-0.50) circle (2pt)node[anchor=west] {$x_{15}$};
%\draw[dashed](2,1) -- (2,0);
\draw[dashed](2,0) -- (2,-1);
%\draw[dashed](2,-1) -- (2,-2);
%\draw[dashed](3,0.5) -- (3,-0.5);
%
\draw[ultra thick] (2,0) -- (1,-0.5)  ;
\draw[ultra thick]  (1,-0.5)-- (2,-1)  ;
%\draw[ultra thick] (2,0) -- (3,-0.5)  ;
%\draw[ultra thick] (3,-0.5) --  (2,-1) ;
\draw[ultra thick](0,1) node[]{$b)$};

\draw (0,-2.04) node[]{ };

\draw[ultra thick] (2,-1) --(2.5,-1) ;
\draw (2.9,-1.04) node[]{$ \ldots $ };

\draw (3,0) node[]{ };
\end{tikzpicture}
\begin{tikzpicture}

\filldraw[black] (2,1) circle (2pt)node[anchor=east] {$y_{1}$};
\filldraw[black] (2,0) circle (2pt)node[anchor=east] {$y_{2}$};
\filldraw[black] (2,-1) circle (2pt)node[anchor=east] {$y_{3}$};
\filldraw[black] (2,-2) circle (2pt)node[anchor=east] {$y_{4}$};
\filldraw[black] (3,0.5) circle (2pt)node[anchor=west] {$x_{1}$};
\filldraw[black] (3,-0.5) circle (2pt)node[anchor=west] {$x_{2}$};
\filldraw[black] (3,-1.5) circle (2pt)node[anchor=west] {$x_0$};
\draw[dashed](2,1) -- (2,0);
\draw[dashed](2,0) -- (2,-1);
\draw[dashed](2,-1) -- (2,-2);
\draw[dashed](3,0.5) -- (3,-0.5);

\draw[ultra thick] (2,1) -- (3,0.5)  ;
\draw[ultra thick]  (3,0.5)-- (2,0)  ;
\draw[ultra thick] (2,0) -- (3,-0.5)  ;
\draw[ultra thick] (3,-0.5) --  (2,-1) ;
\draw[ultra thick]  (2,-1)--  (3,-1.5) ;
\draw[ultra thick] (3,-1.5) -- (2,-2)   ;
\draw (0,0.96) node[]{$c)$};

\draw[ultra thick] (2,-2) --(2.5,-2) ;
\draw (2.9,-2.04) node[]{$ \ldots $ };

\draw (3,0) node[]{ };
\draw (0,-2.04) node[]{ };

\draw (3,0) node[]{ };
\end{tikzpicture}

\end{center}

 \caption{The shared edges of the Hamiltonian paths according to the residual part of $G$.}
\label{residual}

\end{figure}

Let $R$ be the subgraph of $G$ induced by the residual part and the apex. Observe that the already constructed parts of the Hamiltonian paths are $R$-type.

Now we direct our attention to the ladders of $G$. Fix an ordering of the ladders. We construct Hamiltonian paths that correspond to $0-1$ sequences of length $l$. Each Hamiltonian path visits (contains the weight $1$ edges of) the ladders in the prescribed order. (When a Hamiltonian path visits a ladder it will traverse all its weight $1$ edges before visiting the next ladder.)
When a Hamiltonian path visits the $i$th ladder on the vertices $\{v_{1}^{i},\ldots, v_{k_i}^{i},w_{1}^{i},\ldots,w_{k_i}^{i}\}$ where the top of the ladder is the pair $(v_{1}^{i},w_{1}^{i})$, it chooses from two possible paths: $v_{1}^{i}w_{1}^{i}v_{2}^{i}w_{2}^{i} \ldots v_{k_i}^{i}w_{k_i}^{i}$ or $w_{1}^{i}v_{1}^{i}w_{2}^{i}v_{2}^{i} \ldots, w_{k}^{i}v_{k_i}^{i}$ according to the $i$th coordinate of the $0-1$ sequence, see Figure~\ref{more general} (where the single indexed $x_j$'s represent the vertices $v_{r}^{i}$ and $w_{r}^{i}$).

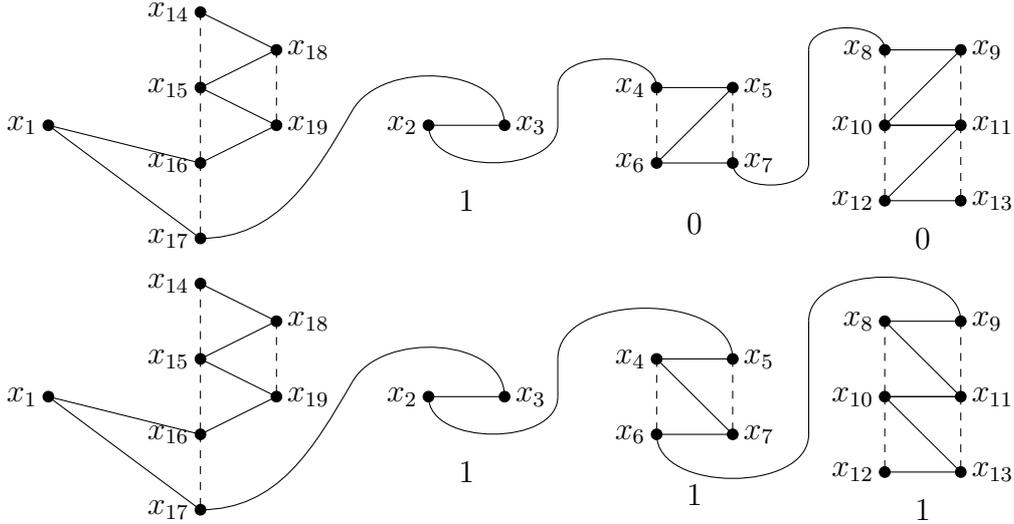
\begin{figure}[htbp,scale=0.5]
\begin{center}
\begin{tikzpicture}
\filldraw[black] (0,-0.5) circle (2pt)node[anchor=east] {$x_1$};

\filldraw[black] (2,1) circle (2pt)node[anchor=east] {$x_{14}$};
\filldraw[black] (2,0) circle (2pt)node[anchor=east] {$x_{15}$};
\filldraw[black] (2,-1) circle (2pt)node[anchor=east] {$x_{16}$};
\filldraw[black] (2,-2) circle (2pt)node[anchor=east] {$x_{17}$};
\filldraw[black] (3,0.5) circle (2pt)node[anchor=west] {$x_{18}$};
\filldraw[black] (3,-0.5) circle (2pt)node[anchor=west] {$x_{19}$};
\draw[dashed](2,1) -- (2,0);
\draw[dashed](2,0) -- (2,-1);
\draw[dashed](2,-1) -- (2,-2);
\draw[dashed](3,0.5) -- (3,-0.5);

\draw (2,1) -- (3,0.5)  ;
\draw  (3,0.5)-- (2,0)  ;
\draw (2,0) -- (3,-0.5)  ;
\draw (3,-0.5) --  (2,-1) ;
\draw  (2,-1)--  (0,-0.5) ;
\draw (0,-0.5) -- (2,-2)   ;

\filldraw[black] (5,-0.5) circle (2pt)node[anchor=east] {$x_2$};
\filldraw[black] (6,-0.5) circle (2pt)node[anchor=west] {$x_3$};
\draw (5,-0.5)--(6,-0.5);
\draw (5.5,-1.5) node[]{$1$};

\filldraw[black] (8,0) circle (2pt)node[anchor=east] {$x_4$};
\filldraw[black] (9,0) circle (2pt)node[anchor=west] {$x_5$};
\filldraw[black] (8,-1) circle (2pt)node[anchor=east] {$x_6$};
\filldraw[black] (9,-1) circle (2pt)node[anchor=west] {$x_7$};
\draw (8,0)--(9,0);
\draw (8,-1)-- (9,-1);
\draw[dashed](8,0) -- (8,-1);
\draw[dashed](9,0) -- (9,-1);
\draw (8.5,-1.8) node[]{$0$};

%\filldraw[black] (11,-0.5) circle (2pt)node[anchor=east] {$x_8$};
%\filldraw[black] (12,-0.5) circle (2pt)node[anchor=west] {$x_9$};
\draw (11,-0.5)--(12,-0.5);
\draw (11.5,-2) node[]{$0$};

\draw  (2,-2) to[out=0,in=-120]  (4,-0.3)  ;
\draw  (4,-0.3) to[out=60,in=90]   (6,-0.5)  ;
\draw  (5,-0.5) to[out=-90,in=-90]  (6.7,-0.5)  ;
\draw   (6.7,-0.5)to[out=90,in=-90]  (6.7,0)  ;
\draw  (6.7,0) to[out=90,in=90]  (8,0)  ;

\draw (9,0)-- (8,-1);
\draw  (9,-1) to[out=-90,in=-90]  (10,-1)  ;
\draw  (10,-1) to[out=90,in=-90]  (10,0.5)  ;
\draw  (10,0.5) to[out=90,in=90]  (11,0.5)  ;

\draw (11,0.5)-- (12,0.5);
\draw (11,-0.5)-- (12,-0.5);
\draw (11,-1.5)-- (12,-1.5);

\draw[dashed](11,0.5) -- (11,-0.5);
\draw[dashed] (11,-0.5)-- (11,-1.5) ;
\draw[dashed] (12,0.5)-- (12,-0.5);
\draw[dashed]  (12,-0.5)-- (12,-1.5);

\filldraw[black] (11,0.5) circle (2pt)node[anchor=east] {$x_8$};
\filldraw[black] (11,-0.5) circle (2pt)node[anchor=east] {$x_{10}$};
\filldraw[black] (11,-1.5) circle (2pt)node[anchor=east] {$x_{12}$};

\filldraw[black] (12,0.5) circle (2pt)node[anchor=west] {$x_9$};
\filldraw[black] (12,-0.5) circle (2pt)node[anchor=west] {$x_{11}$};
\filldraw[black] (12,-1.5) circle (2pt)node[anchor=west] {$x_{13}$};

\draw (12,0.5) --(11,-0.5) ;
\draw (12,-0.5) -- (11,-1.5) ;

\end{tikzpicture}

\begin{tikzpicture}
\filldraw[black] (0,-0.5) circle (2pt)node[anchor=east] {$x_1$};

\filldraw[black] (2,1) circle (2pt)node[anchor=east] {$x_{14}$};
\filldraw[black] (2,0) circle (2pt)node[anchor=east] {$x_{15}$};
\filldraw[black] (2,-1) circle (2pt)node[anchor=east] {$x_{16}$};
\filldraw[black] (2,-2) circle (2pt)node[anchor=east] {$x_{17}$};
\filldraw[black] (3,0.5) circle (2pt)node[anchor=west] {$x_{18}$};
\filldraw[black] (3,-0.5) circle (2pt)node[anchor=west] {$x_{19}$};
\draw[dashed](2,1) -- (2,0);
\draw[dashed](2,0) -- (2,-1);
\draw[dashed](2,-1) -- (2,-2);
\draw[dashed](3,0.5) -- (3,-0.5);

\draw (2,1) -- (3,0.5)  ;
\draw  (3,0.5)-- (2,0)  ;
\draw (2,0) -- (3,-0.5)  ;
\draw (3,-0.5) --  (2,-1) ;
\draw  (2,-1)--  (0,-0.5) ;
\draw (0,-0.5) -- (2,-2)   ;

\filldraw[black] (5,-0.5) circle (2pt)node[anchor=east] {$x_2$};
\filldraw[black] (6,-0.5) circle (2pt)node[anchor=west] {$x_3$};
\draw (5,-0.5)--(6,-0.5);
\draw (5.5,-1.5) node[]{$1$};

\filldraw[black] (8,0) circle (2pt)node[anchor=east] {$x_4$};
\filldraw[black] (9,0) circle (2pt)node[anchor=west] {$x_5$};
\filldraw[black] (8,-1) circle (2pt)node[anchor=east] {$x_6$};
\filldraw[black] (9,-1) circle (2pt)node[anchor=west] {$x_7$};
\draw (8,0)--(9,0);
\draw (8,-1)-- (9,-1);
\draw[dashed](8,0) -- (8,-1);
\draw[dashed](9,0) -- (9,-1);
\draw (8.5,-1.8) node[]{$1$};

%\filldraw[black] (11,-0.5) circle (2pt)node[anchor=east] {$x_8$};
%\filldraw[black] (12,-0.5) circle (2pt)node[anchor=west] {$x_9$};
\draw (11,-0.5)--(12,-0.5);
\draw (11.5,-2) node[]{$1$};

\draw  (2,-2) to[out=0,in=-120]  (4,-0.3)  ;
\draw  (4,-0.3) to[out=60,in=90]   (6,-0.5)  ;
%\draw  (5,-0.5) to[out=-90,in=-110]  (6.7,-0.5)  ;
%\draw  (6.7,-0.5) to[out=70,in=90]  (9,0)  ;
\draw (8,0)-- (9,-1);
%\draw  (8,-1) to[out=-90,in=-120]  (10,-0.7)  ;
%\draw  (10,-0.7) to[out=60,in=90]  (12,-0.5)  ;

\draw  (5,-0.5) to[out=-90,in=-90]  (6.7,-0.5)  ;
\draw   (6.7,-0.5)to[out=90,in=-90]  (6.7,0)  ;
\draw  (6.7,0) to[out=90,in=90]  (9,0)  ;

\draw  (8,-1) to[out=-90,in=-90]  (10,-1)  ;
\draw  (10,-1) to[out=90,in=-90]  (10,0.5)  ;
\draw  (10,0.5) to[out=90,in=90]  (12,0.5)  ;

\draw (11,0.5)-- (12,0.5);
\draw (11,-0.5)-- (12,-0.5);
\draw (11,-1.5)-- (12,-1.5);

\filldraw[black] (11,0.5) circle (2pt)node[anchor=east] {$x_8$};
\filldraw[black] (11,-0.5) circle (2pt)node[anchor=east] {$x_{10}$};
\filldraw[black] (11,-1.5) circle (2pt)node[anchor=east] {$x_{12}$};

\filldraw[black] (12,0.5) circle (2pt)node[anchor=west] {$x_9$};
\filldraw[black] (12,-0.5) circle (2pt)node[anchor=west] {$x_{11}$};
\filldraw[black] (12,-1.5) circle (2pt)node[anchor=west] {$x_{13}$};

\draw[dashed](11,0.5) -- (11,-0.5);
\draw[dashed] (11,-0.5)-- (11,-1.5) ;
\draw[dashed] (12,0.5)-- (12,-0.5);
\draw[dashed]  (12,-0.5)-- (12,-1.5);

\draw (11,0.5) --(12,-0.5) ;
\draw (11,-0.5) -- (12,-1.5) ;

\end{tikzpicture}
 \caption{The paths that satisfy a ladder contain many parts that resemble a "Z" or a reversed "Z" shape, according to the choice made at the top, hence the name Z-swapping construction.}
\label{more general}
\end{center}
\end{figure}

Clearly, every Hamiltonian path constructed this way is $G$-type. These Hamiltonian paths are compatible:  For two such Hamiltonian paths, let $i$ be the first coordinate where their $0-1$ sequences differ. Until the bottom of the $(i-1)$th ladder the two paths consist of the same edges. There is a triangle that consists of the last vertex of the paths at the bottom of the $(i-1)$th ladder and the two vertices at the top of the $i$th ladder (see Figure~\ref{more general}).
 Thus we have constructed as many pairwise compatible $G$-type Hamiltonian paths as many different $0-1$ sequences of length $l$ exist and this completes the proof.

\end{proof}

\begin{remark}
Although we will not need this, we mention that it is not hard to prove that the size $2^l$ in Lemma \ref{swapping} is best possible. Here is a sketch of the proof. A Hamiltonian path is a bipartite graph. Observe that the vertices on one side of a ladder must be on the same side of the bipartition, and the vertices on the other side of the same ladder must be on the other side of the bipartition. Since a Hamiltonian path is a balanced bipartite graph, the vertices of the residual part must be distributed in such a way, that the longer path is on one side, and the shorter path plus the apex is on the other, this distinguishes one side of the bipartition. Thus our only possibility to construct $G$-type  Hamiltonian paths for a properly laddered $G$ with different bipartitions is to "swap" the sides of the ladders between the sides of the bipartition. Which is exactly what happens in the proof of Lemma \ref{swapping}.
\end{remark}

%\begin{remark}
%This construction works for any ordering of the ladders. Let $\mathcal{G}_n$ be the graph with the vertices the Hamiltonian paths on $n$ vertices, and two such vertices are connected if the Hamiltonian paths are compatible. Lemma~\ref{swapping} not only guarantees a clique of size $2^l$ in $\mathcal{G}_n$, but actually $l!$ (usually overlapping) cliques. Determining the maximal number of compatible Hamiltonian paths is equivalent to determine $\omega (\mathcal{G}_n)$, which we do in Theorem~\ref{main}. Since we will use Lemma~\ref{swapping} to prove Theorem~\ref{main}, we will actually construct many maximal cliques $\mathcal{G}_n$. 
%\end{remark}

For a weighted graph $G$, we refer to the construction of Hamiltonian paths in Lemma~\ref{swapping} as applying the Z-swapping construction to $G$. We also apply the Z-swapping construction to a family of weighted graphs and by this we mean that we apply it to each weighted graph in the family and we take the union of the resulting sets of compatible Hamiltonian paths.

Now we already know how to construct Hamiltonian paths from weighted graphs, thus we are left with the task of building a compatible family of weighted graphs from which we can construct the right number of  Hamiltonian paths. From now on we will mainly work with weighted graphs. %First we will define two families of weighted graphs with numerous properties, than we show that the existence of such families is enough to prove Theorem~\ref{main}. Finally we present a procedure to construct such families.

\begin{defi}
We say that a compatible family $\mathcal{F}$ of properly laddered weighted graphs on $2n+1$ vertices is H-maximal if applying the Z-swapping construction to $\mathcal{F}$ we get $\binom{2n+1}{n}$ Hamiltonian paths.
\end{defi}

Thus we can get the maximum possible number of Hamiltonian paths from a H-maximal family using the Z-swapping construction. Our goal is to build H-maximal families. To do this we define a similar family which satisfies an additional condition.

\begin{defi}
We say that a compatible family $\mathcal{F}$ of properly laddered weighted graphs on a ground set of size $2k+1$ is MH-maximal if there is a matching $M$ of size $k$ that is contained in every weighted graph in $\mathcal{F}$, such that every edge of $M$ gets weight $2$ in every weighted graph and applying the Z-swapping construction to $\mathcal{F}$ we get $\binom{k}{\left \lfloor \frac{k}{2} \right \rfloor}$ Hamiltonian paths.
\end{defi}

%\begin{defi}
%We say that a set of graphs $\{G_i\}_{i=1}^{m}$ from $ \mathcal{G}$ on a ground set of size $2k+1$ is an auxiliary Z-family if the following holds. Every pair of graphs from $\{G_i\}_{i=1}^{m}$ is compatible. There exists a matching $M$ of size $k$ such that every $G_i$ contains the edges of $M$ with weight $2$. Moreover if the number of ladders in $G_i$ is exactly $l_i$, then $\sum_{i=1}^{m}2^{l_i}= \binom{k}{\left\lfloor \frac{k}{2} \right\rfloor}$.
%\end{defi}

 Finding MH-maximal families for small vertex sets is not difficult, see Figure~\ref{supportZexists}, but it is nontrivial whether they exist for all odd-element vertex sets. (We will prove that they do in Lemma~\ref{finish}.)

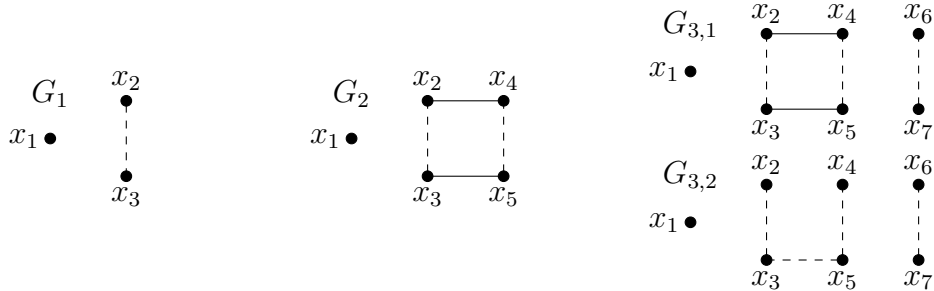
\begin{figure}[htbp,scale=0.5]
\begin{center}
\begin{tikzpicture}
\filldraw[black] (0,0) circle (2pt)node[anchor=east] {$x_1$};
\filldraw[black] (1,0.5) circle (2pt)node[anchor=south] {$x_2$};
\filldraw[black] (1,-0.5) circle (2pt)node[anchor=north] {$x_3$};
\draw[dashed](1,-0.5)--(1,0.5);
\draw (0,0.6) node[]{$G_1$};
\draw(3,-2) node[]{};
\end{tikzpicture}
\begin{tikzpicture}
\filldraw[black] (5,0) circle (2pt)node[anchor=east] {$x_1$};
\filldraw[black] (6,0.5) circle (2pt)node[anchor=south] {$x_2$};
\filldraw[black] (6,-0.5) circle (2pt)node[anchor=north] {$x_3$};
\filldraw[black] (7,0.5) circle (2pt)node[anchor=south] {$x_4$};
\filldraw[black] (7,-0.5) circle (2pt)node[anchor=north] {$x_5$};
\draw[dashed](6,0.5)--(6,-0.5);
\draw[dashed](7,0.5)--(7,-0.5);
\draw(6,-0.5)--(7,-0.5);
\draw(6,0.5)--(7,0.5);
\draw (5,0.6) node[]{$G_2$};
\draw(8.5,-2) node[]{};
\end{tikzpicture}
\begin{tikzpicture}
\filldraw[black] (5,0) circle (2pt)node[anchor=east] {$x_1$};
\filldraw[black] (6,0.5) circle (2pt)node[anchor=south] {$x_2$};
\filldraw[black] (6,-0.5) circle (2pt)node[anchor=north] {$x_3$};
\filldraw[black] (7,0.5) circle (2pt)node[anchor=south] {$x_4$};
\filldraw[black] (7,-0.5) circle (2pt)node[anchor=north] {$x_5$};
\filldraw[black] (8,0.5) circle (2pt)node[anchor=south] {$x_6$};
\filldraw[black] (8,-0.5) circle (2pt)node[anchor=north] {$x_7$};
\draw[dashed](6,0.5)--(6,-0.5);
\draw[dashed](7,0.5)--(7,-0.5);
\draw[dashed](8,0.5)--(8,-0.5);
\draw(6,-0.5)--(7,-0.5);
\draw(6,0.5)--(7,0.5);
\draw (5,0.6) node[]{$G_{3,1}$};

\filldraw[black] (5,-2) circle (2pt)node[anchor=east] {$x_1$};
\filldraw[black] (6,-1.5) circle (2pt)node[anchor=south] {$x_2$};
\filldraw[black] (6,-2.5) circle (2pt)node[anchor=north] {$x_3$};
\filldraw[black] (7,-1.5) circle (2pt)node[anchor=south] {$x_4$};
\filldraw[black] (7,-2.5) circle (2pt)node[anchor=north] {$x_5$};
\filldraw[black] (8,-1.5) circle (2pt)node[anchor=south] {$x_6$};
\filldraw[black] (8,-2.5) circle (2pt)node[anchor=north] {$x_7$};
\draw[dashed](6,-1.5)--(6,-2.5);
\draw[dashed](7,-1.5)--(7,-2.5);
\draw[dashed](8,-1.5)--(8,-2.5);
\draw[dashed](6,-2.5)--(7,-2.5);
\draw (5,-1.4) node[]{$G_{3,2}$};
\end{tikzpicture}
 \caption{For $k=1,2$ an MH-maximal family consists of a single weighted graph: $G_1$ and $G_2$ where $M$ is $\{(x_2,x_3)\}$ and $\{(x_2,x_3),(x_4,x_5)\}$ respectively. For $k=3$ the weighted graphs $G_{3,1},G_{3,2}$ form an MH-maximal family where the matching $M$ is $\{(x_2,x_3),(x_4,x_5),(x_6,x_7)\}$.}
\label{supportZexists}
\end{center}
\end{figure}

\begin{remark}
%It can be proved that MH-maximal families provide the maximal possible number of Hamiltonian paths in the following sense.
Let $\mathcal{F}$ be an MH-maximal family on $2k+1$ vertices with the corresponding matching $M$. Observe that every Hamiltonian path that the Z-swapping construction produces from $\mathcal{F}$ is $M$-type. It can be proven that the maximal possible number of $M$-type Hamiltonian paths is also $\binom{k}{\left \lfloor \frac{k}{2} \right \rfloor}$, hence the name MH-maximal. Since we will not use this fact, we only sketch the argument. Hamiltonian paths are balanced bipartite graphs. Observe that two vertices that are connected by an edge of weight two in a weighted graph $G$ must be on the same side of the bipartition for a $G$-type Hamiltonian path. Thus the vertices that are connected by an edge of $M$ are "glued together" so we are essentially interested in balanced bipartitions on $|M|=k$ vertices.

%Thus one can think of MH-complete families as families that give us the optimal number of Hamiltonian paths with the additional requirement that the vertices connected by $M$ should be on the same side of the bipartition.
\end{remark}

The following lemma states that we can build H-maximal families using MH-maximal ones.

\begin{lem} \label{MH->H}
%If there is a $Z$-family on $2n+1$ vertices then there is a set of $\binom{2n+1}{n}$ compatible Hamiltonian paths on the same vertex set.
If there exist MH-maximal families on ground sets of size $3,5,\ldots,2n+1$ then there is a H-maximal family $\mathcal{F}$ on a ground set of size $2n+1$.
\end{lem}
\begin{proof}
Let the ground set be $\{x_1,x_2,\ldots, x_{2n+1}\}$ and $B$ be the matching that consists of the edges $(x_2,x_3),(x_4,x_5), \ldots, (x_{2n},x_{2n+1})$. For each submatching $M \subseteq B$, we define a family $\mathcal{F}_M$ of properly laddered weighted graphs as follows. We put an MH-maximal family on the vertices of $M$ and the vertex $x_1$ using $M$ as the corresponding matching in the MH-maximal family and $x_1$ as the apex. Then we extend every graph of this MH-maximal family by adding the edges of $B \setminus M$ with weight $1$. Note that a single edge of weight $1$ is a special ladder (a $1$-ladder, in particular), thus $\mathcal{F}_M$ consists of properly laddered weighted graphs. Now let
$$\mathcal{F}=\bigcup_{M \subseteq B} \mathcal{F}_M.$$
We will prove that $\mathcal{F}$ is a H-maximal family.

The graphs in the family $\mathcal{F}$ are compatible by the following argument. Two graphs that correspond to different submatchings $M_1,M_2$ are compatible, since they contain every edge in the symmetric difference $M_1 \triangle M_2$ with different weights. Two graphs that correspond to the same matching are compatible, since they are built from an MH-maximal family which consists of compatible weighted graphs.

We show that we get the desired number of Hamiltonian paths by applying the Z-swapping construction to $\mathcal{F}$. First we count the number of Hamiltonian paths that we can construct from each $\mathcal{F}_M$. Let $|M|=k$, the graphs in $\mathcal{F}_M$ are constructed from an MH-maximal family on a ground set of size $2k+1$ by adding exactly $n-k$ edges (as $1$-ladders) to every graph. By definition we get $\binom{k}{\left \lfloor \frac{k}{2} \right \rfloor }$ Hamiltonian paths when applying the Z-swapping construction to an MH-maximal family. Since each graph in $\mathcal{F}_M$ has $n-k$ additional ladders, by Lemma \ref{swapping}, we get exactly $\binom{k}{\left\lfloor \frac{k}{2} \right\rfloor} 2^{n-k}$ Hamiltonian paths by applying the Z-swapping construction to $\mathcal{F}_M$. Thus applying the Z-swapping construction to $\mathcal{F}$, we can get exactly

$$ \sum_{k=0}^n \binom{n}{k}\binom{k}{\left\lfloor \frac{k}{2} \right\rfloor}2^{n-k} $$

Hamiltonian paths. We are left with the task of proving the following combinatorial identity:
 \[ \sum_{k=0}^n \binom{n}{k}\binom{k}{\left\lfloor \frac{k}{2} \right\rfloor}2^{n-k}=\binom{2n+1}{n}. \]

\if
The coefficient of $x^n$ in the expression $(1+x)^{2n+1}$ is clearly the right hand side. We will expand $(1+x)^{2n+1}$ in such a way that the coefficient of $x^n$ will become the left hand side.

\begin{gather*}
(1+x)^{2n+1}=(1+x)(1+2x+x^2)^{n} =(1+x)((1+x^2)+2x)^{n} = \\
 (1+x)\sum_{k=0}^{n}\binom{n}{k}(1+x^2)^{k}(2x)^{n-k} =
 (1+x)\sum_{k=0}^{n}\binom{n}{n-k}(1+x^2)^{k}(2x)^{n-k}= \\
 (1+x)\sum_{k=0}^{n}\sum_{j=0}^{k}\binom{n}{n-k}\binom{k}{j}x^{2j}(2x)^{n-k}=
 (1+x)\sum_{k=0}^{n}\sum_{j=0}^{k}\binom{n}{n-k}\binom{k}{j}2^{n-k}x^{2j+n-k}=\\
 \sum_{k=0}^{n}\sum_{j=0}^{k}\binom{n}{n-k}\binom{k}{j}2^{n-k}x^{2j+n-k}+\sum_{k=0}^{n}\sum_{j=0}^{k}\binom{n}{n-k}\binom{k}{j}2^{n-k}x^{2j+n-k+1}
\end{gather*}

Here we get the coefficient of $x^n$ in the first sum if $k$ is even and $j=k/2$, and in the second sum when $k$ is odd and $j=(k-1)/2$
\begin{gather*}
\sum_{\substack{k=0 \\ k \text{ even}}}^{n}\binom{n}{n-k}\binom{k}{\frac{k}{2}}2^{n-k}+\sum_{\substack{k=0 \\ k \text{ odd}}}^{n}\binom{n}{n-k}\binom{k}{\frac{k-1}{2}}2^{n-k}= \sum_{k=0}^{n}\binom{n}{n-k}\binom{k}{\left\lfloor \frac{k}{2} \right\rfloor }2^{n-k}
\end{gather*}
completing the proof.
\fi

The following short argument is due to G\'eza T\'oth.

Observe that the left hand side is equal to the number of $4$-partitions ${\it P}=(A_0,A_1,B,C)$ of the set $[n]=\{1,\dots,n\}$, where we require $0\le |A_1|-|A_0|\le 1$.
%(The running parameter $k$ belongs to $|A_0\cup A_1|=|A_0|+|A_1|$.)
For each such ${\it P}$ attach the set $D({\it P}):=A_0\cup B\cup \{-i: 1\le i\le n, i\in (A_0\cup C)\}$. Let $\overline{D}({\it P})=D({\it P})$ if $|A_1|=|A_0|$ and $\overline{D}({\it P})=D({\it P})\cup \{0\}$ if $|A_1|=|A_0|+1$. Then $\overline{D}({\it P})$ is an $n$-element subset of the $(2n+1)$-element set $\{-n,\dots,-1,0,1,\dots,n\}$ and every such $n$-element subset belongs to exactly one $4$-partition ${\it P}$ of the above type. This implies that the number of these $4$-partitions is exactly $\binom{2n+1}{n}$ proving the identity and thus completing the proof of the lemma.
\end{proof}

\begin{remark}
Kitti Varga gave a different proof of the combinatorial identity using polynomials. We only sketch her proof. Start with
\[ (1+x)^{2n+1}=(1+x)(1+2x+x^2)^{n} =(1+x)((1+x^2)+2x)^{n}. \]
Expand the right hand side using the binomial theorem twice. Then comparing the coefficient of $x^{n}$ of the left hand side and the expanded right hand side yields the desired identity.
\end{remark}

\medskip

Now we see that to prove Theorem~\ref{main} it is enough to build MH-maximal families. Thus the next lemma provides what we still need to finish the proof of Theorem~\ref{main}.

\begin{lem}\label{finish}
For every positive integer $k$ there is an MH-maximal family on $2k+1$ vertices.
\end{lem}
\begin{proof}
We will build MH-maximal families using H-maximal families on smaller ground sets. We proceed with induction on $k$. 
We have already seen the existence of MH-maximal families for $k=1,2,3$ (see Figure~\ref{supportZexists}), so the base case is proven. Now suppose that there exist MH-maximal families on $1,3,\ldots, 2k-1$ vertices, and we build one on $2k+1$ vertices. We will have two separate induction steps, one when $k$ is odd and a little different one when $k$ is even.

\medskip \textbf{If $k$ is odd:} Since $k$ is odd and less than $2k-1$, by the induction hypothesis there are MH-maximal families for each odd number up to $k$. Thus by Lemma~\ref{MH->H} there is a H-maximal family on a ground set of size $k$. By definition we can construct $\binom{k}{\left\lfloor \frac{k}{2} \right\rfloor}$ Hamiltonian paths from this H-maximal family. Observe that this is exactly the number of Hamiltonian paths that is required in the definition of the MH-maximal family on $2k+1$ vertices! To obtain this same number, we will transform our H-maximal family into an MH-maximal family by preserving the number of graphs and the number of ladders for each graph. By Lemma~\ref{swapping} this ensures that we can construct the exact same number of Hamiltonian paths from the family after the transformation.

 We will transform each weighted graph $G$ on the vertices $\{x_1, \ldots, x_k\}$ to a weighted graph $G'$ on the vertices $\{w, x_1,x_2,\ldots,x_k,x'_1,x'_2,\ldots,x'_k\}$ in such a way, that the matching $M=\{(x_1,x'_1),(x_2,x'_2), \ldots , (x_k, x'_k)\}$ is a subgraph of $G'$, every edge of $M$ gets weight $2$ and the vertex $w$ will serve as the apex of $G'$.

We transform each component of $G$ separately: A weighted ladder on $h$ vertices is transformed into a weighted ladder of $2h$ vertices as follows. If $(x_i,x_j)$ is an edge of weight $1$ in $G$ then in $G'$, both $(x_i,x_j)$ and $(x'_i,x'_j)$ are edges of weight $1$. If $(x_i,x_j)$ was an edge of weight $2$ in $G$, then in $G'$ we alternately choose $(x_i,x_j)$ or $(x'_i,x'_j)$ to be an edge of weight $2$ in $G'$ starting from the top of the ladder and doing differently "above" and "below" each weight $1$ edge, see Figure~\ref{ugly part}.
%we connect either $(x_i,x_j)$ or $(x'_i,x'_j)$ by an edge of weight $2$ alternately starting from the top of the ladder, see Figure~\ref{ugly part}.

\begin{figure}[htbp,scale=0.5]
\begin{center}
\begin{tikzpicture}
\tikzstyle{vertex}=[draw,circle,fill=black,minimum size=4,inner sep=0]

\node[vertex] (x2) at (0,0) [label=left:$x_2$] {};
\node[vertex] (x3) at (1,0) [label=right:$x_3$] {};
\draw (x2) -- (x3);

\begin{scope}[shift={(5,0)}]
\node[vertex] (x2) at (0,0) [label=below:$x_2$] {};
\node[vertex] (x3) at (1,0) [label=below:$x_3$] {};
\node[vertex] (x2') at (-1,0) [label=left:$x_2'$] {};
\node[vertex] (x3') at (2,0) [label=right:$x_3'$] {};
\draw (x2) -- (x3);
\draw[dashed] (x2) -- (x2');
\draw[dashed] (x3) -- (x3');
\draw (x2') to [bend left] (x3');
\end{scope}
\end{tikzpicture}
\end{center}

\begin{center}
\begin{tikzpicture}
\tikzstyle{vertex}=[draw,circle,fill=black,minimum size=4,inner sep=0]

\node[vertex] (x2) at (0,1) [label=left:$x_2$] {};
\node[vertex] (x3) at (1,1) [label=right:$x_3$] {};
\node[vertex] (x4) at (0,0) [label=left:$x_4$] {};
\node[vertex] (x5) at (1,0) [label=right:$x_5$] {};
\draw (x2) -- (x3);
\draw (x4) -- (x5);
\draw[dashed] (x2) -- (x4);
\draw[dashed] (x3) -- (x5);

\begin{scope}[shift={(5,0)}]
\node[vertex] (x2) at (0,1) [label=below:$x_2$] {};
\node[vertex] (x3) at (1,1) [label=below:$x_3$] {};
\node[vertex] (x4) at (0,0) [label=below:$x_4$] {};
\node[vertex] (x5) at (1,0) [label=below:$x_5$] {};
\node[vertex] (x2') at (-1,1) [label=left:$x_2'$] {};
\node[vertex] (x3') at (2,1) [label=right:$x_3'$] {};
\node[vertex] (x4') at (-1,0) [label=left:$x_4'$] {};
\node[vertex] (x5') at (2,0) [label=right:$x_5'$] {};

\draw (x2) -- (x3);
\draw (x4) -- (x5);
\draw[dashed] (x2) -- (x2');
\draw[dashed] (x3) -- (x3');
\draw[dashed] (x4) -- (x4');
\draw[dashed] (x5) -- (x5');
\draw[dashed] (x2') -- (x4');
\draw[dashed] (x3') -- (x5');
\draw (x2') to [bend left] (x3');
\draw (x4') to [bend left] (x5');
\end{scope}
\end{tikzpicture}
\end{center}

\begin{center}
\begin{tikzpicture}
\tikzstyle{vertex}=[draw,circle,fill=black,minimum size=4,inner sep=0]

\node[vertex] (x2) at (0,1) [label=left:$x_2$] {};
\node[vertex] (x3) at (1,1) [label=right:$x_3$] {};
\node[vertex] (x4) at (0,0) [label=left:$x_4$] {};
\node[vertex] (x5) at (1,0) [label=right:$x_5$] {};
\node[vertex] (x6) at (0,-1) [label=left:$x_6$] {};
\node[vertex] (x7) at (1,-1) [label=right:$x_7$] {};
\draw (x2) -- (x3);
\draw (x4) -- (x5);
\draw (x6) -- (x7);
\draw[dashed] (x2) -- (x6);
\draw[dashed] (x3) -- (x7);

\begin{scope}[shift={(5,0)}]
\node[vertex] (x2) at (0,1) [label=below:$x_2$] {};
\node[vertex] (x3) at (1,1) [label=below:$x_3$] {};
\node[vertex] (x4) at (0,0) [label=below left:$x_4$] {};
\node[vertex] (x5) at (1,0) [label=below right:$x_5$] {};
\node[vertex] (x2') at (-1,1) [label=left:$x_2'$] {};
\node[vertex] (x3') at (2,1) [label=right:$x_3'$] {};
\node[vertex] (x4') at (-1,0) [label=left:$x_4'$] {};
\node[vertex] (x5') at (2,0) [label=right:$x_5'$] {};

\node[vertex] (x6) at (0,-1) [label=below:$x_6$] {};
\node[vertex] (x7) at (1,-1) [label=below:$x_7$] {};
\node[vertex] (x6') at (-1,-1) [label=left:$x_6'$] {};
\node[vertex] (x7') at (2,-1) [label=right:$x_7'$] {};

\draw (x2) -- (x3);
\draw (x4) -- (x5);
\draw[dashed] (x2) -- (x2');
\draw[dashed] (x3) -- (x3');
\draw[dashed] (x4) -- (x4');
\draw[dashed] (x5) -- (x5');
\draw[dashed] (x2') -- (x4');
\draw[dashed] (x3') -- (x5');
\draw (x2') to [bend left] (x3');
\draw (x4') to [bend left] (x5');

\draw (x6) -- (x7);
\draw[dashed] (x6) -- (x6');
\draw[dashed] (x7) -- (x7');
\draw (x6') to [bend left] (x7');
\draw[dashed] (x4) -- (x6);
\draw[dashed] (x5) -- (x7);
\end{scope}
\end{tikzpicture}
\end{center}

\begin{center}
\begin{tikzpicture}
\tikzstyle{vertex}=[draw,circle,fill=black,minimum size=4,inner sep=0]

\node[vertex] (x2) at (0,1) [label=left:$x_2$] {};
\node[vertex] (x3) at (1,1) [label=right:$x_3$] {};
\node[vertex] (x4) at (0,0) [label=left:$x_4$] {};
\node[vertex] (x5) at (1,0) [label=right:$x_5$] {};
\node[vertex] (x6) at (0,-1) [label=left:$x_6$] {};
\node[vertex] (x7) at (1,-1) [label=right:$x_7$] {};
\node[vertex] (x8) at (0,-2) [label=left:$x_8$] {};
\node[vertex] (x9) at (1,-2) [label=right:$x_9$] {};
\draw (x2) -- (x3);
\draw (x4) -- (x5);
\draw (x6) -- (x7);
\draw (x8) -- (x9);
\draw[dashed] (x2) -- (x8);
\draw[dashed] (x3) -- (x9);

\begin{scope}[shift={(5,0)}]
\node[vertex] (x2) at (0,1) [label=below:$x_2$] {};
\node[vertex] (x3) at (1,1) [label=below:$x_3$] {};
\node[vertex] (x4) at (0,0) [label=below left:$x_4$] {};
\node[vertex] (x5) at (1,0) [label=below right:$x_5$] {};
\node[vertex] (x2') at (-1,1) [label=left:$x_2'$] {};
\node[vertex] (x3') at (2,1) [label=right:$x_3'$] {};
\node[vertex] (x4') at (-1,0) [label=left:$x_4'$] {};
\node[vertex] (x5') at (2,0) [label=right:$x_5'$] {};

\node[vertex] (x6) at (0,-1) [label=below:$x_6$] {};
\node[vertex] (x7) at (1,-1) [label=below:$x_7$] {};
\node[vertex] (x8) at (0,-2) [label=below:$x_8$] {};
\node[vertex] (x9) at (1,-2) [label=below:$x_9$] {};
\node[vertex] (x6') at (-1,-1) [label=left:$x_6'$] {};
\node[vertex] (x7') at (2,-1) [label=right:$x_7'$] {};
\node[vertex] (x8') at (-1,-2) [label=left:$x_8'$] {};
\node[vertex] (x9') at (2,-2) [label=right:$x_9'$] {};

\draw (x2) -- (x3);
\draw (x4) -- (x5);
\draw[dashed] (x2) -- (x2');
\draw[dashed] (x3) -- (x3');
\draw[dashed] (x4) -- (x4');
\draw[dashed] (x5) -- (x5');
\draw[dashed] (x2') -- (x4');
\draw[dashed] (x3') -- (x5');
\draw (x2') to [bend left] (x3');
\draw (x4') to [bend left] (x5');

\draw (x6) -- (x7);
\draw (x8) -- (x9);
\draw[dashed] (x6) -- (x6');
\draw[dashed] (x7) -- (x7');
\draw[dashed] (x8) -- (x8');
\draw[dashed] (x9) -- (x9');
\draw[dashed] (x6') -- (x8');
\draw[dashed] (x7') -- (x9');
\draw (x6') to [bend left] (x7');
\draw (x8') to [bend left] (x9');

\draw[dashed] (x4) -- (x6);
\draw[dashed] (x5) -- (x7);
\end{scope}
\end{tikzpicture}
\end{center}
\caption{The transformation of the ladders.}
\label{ugly part}
\end{figure}
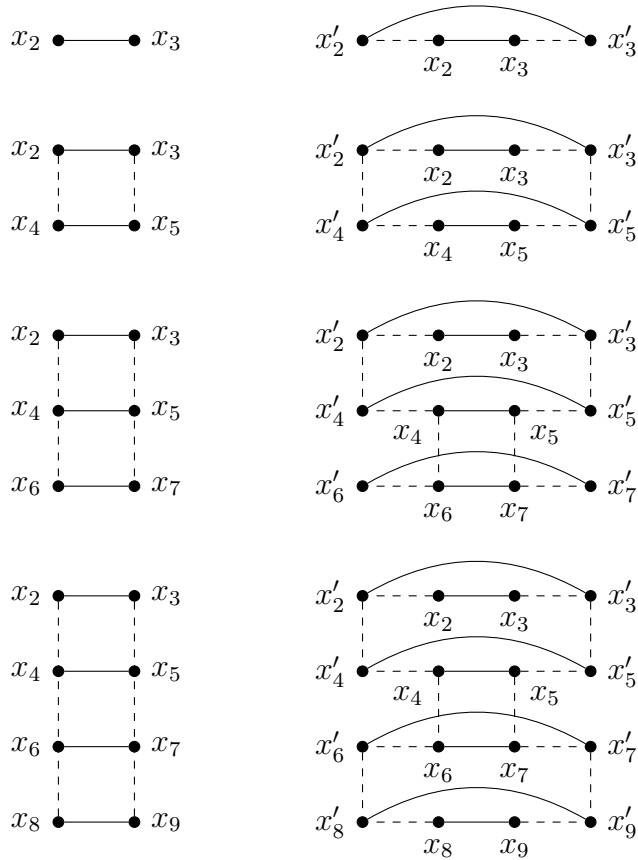

A path that consists of edges of weight $2$ in (the residual part of) $G$ is transformed as follows: If $(x_i,x_j)$ was an edge of weight $2$ in $G$, then in $G'$ either $(x_i,x_j)$ or $(x'_i,x'_j)$ is an edge of weight $2$ so that with the edges of $M$ the resulting component of $G'$ is also a path, see Figure~\ref{paths}.

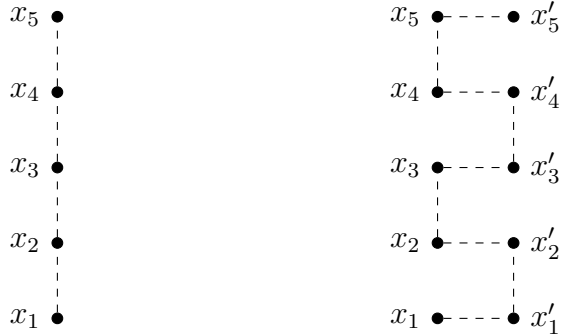
\begin{figure}[htbp,scale=0.5]
\begin{center}
\begin{tikzpicture}
\tikzstyle{vertex}=[draw,circle,fill=black,minimum size=4,inner sep=0]

\node[vertex] (x1) at (0,0) [label=left:$x_1$] {};
\node[vertex] (x2) at (0,1) [label=left:$x_2$] {};
\node[vertex] (x3) at (0,2) [label=left:$x_3$] {};
\node[vertex] (x4) at (0,3) [label=left:$x_4$] {};
\node[vertex] (x5) at (0,4) [label=left:$x_5$] {};
\draw[dashed] (x1) -- (x2);
\draw[dashed] (x2) -- (x3);
\draw[dashed] (x3) -- (x4);
\draw[dashed] (x4) -- (x5);

\begin{scope}[shift={(5,0)}]

\node[vertex] (x1) at (0,0) [label=left:$x_1$] {};
\node[vertex] (x2) at (0,1) [label=left:$x_2$] {};
\node[vertex] (x3) at (0,2) [label=left:$x_3$] {};
\node[vertex] (x4) at (0,3) [label=left:$x_4$] {};
\node[vertex] (x5) at (0,4) [label=left:$x_5$] {};
\node[vertex] (x'1) at (1,0) [label=right:$x'_1$] {};
\node[vertex] (x'2) at (1,1) [label=right:$x'_2$] {};
\node[vertex] (x'3) at (1,2) [label=right:$x'_3$] {};
\node[vertex] (x'4) at (1,3) [label=right:$x'_4$] {};
\node[vertex] (x'5) at (1,4) [label=right:$x'_5$] {};
\draw[dashed] (x1) -- (x'1);
\draw[dashed] (x2) -- (x'2);
\draw[dashed] (x3) -- (x'3);
\draw[dashed] (x4) -- (x'4);
\draw[dashed] (x5) -- (x'5);
\draw[dashed] (x'1) -- (x'2);
\draw[dashed] (x2) -- (x3);
\draw[dashed] (x'3) -- (x'4);
\draw[dashed] (x4) -- (x5);

\end{scope}
\end{tikzpicture}
\end{center}
\caption{The transformation of the paths.}
\label{paths}
\end{figure}

Note that this transformation doubles the number of vertices of the paths in the residual part, thus the difference in the number of their vertices is now four. To repair this we transform the apex of $G$ to a path containing two vertices connected by an edge of weight two and we attach this to the end of the shorter path restoring the length difference of the residual part to two. (Remember that we still have the apex $w$.) Thus this transformation produces properly laddered weighted graphs.

The transformed weighted graphs are compatible by the following argument. Let $G_1$ and $G_2$ be arbitrary weighted graphs from the H-maximal family and $G'_1,G'_2$ their transformed counterparts from the MH-maximal family. Let $(x_i,x_j)$ be (one of) the edge(s) that is contained in both $G_1$ and $G_2$ but with different weight. Without loss of generality we can assume that it gets weight $1$ in $G_1$. We transformed $G_1$ in such a way that $G'_1$ contains both edges $(x_i,x_j)$ and  $(x'_i,x'_j)$ with weight $1$. However, since $G_2$ contains the edge $(x_i,x_j)$ with weight $2$, we transformed it in such a way that in $G'_2$ either the edge $(x_i,x_j)$ or the edge $(x'_i,x'_j)$ gets weight $2$. This takes care of the induction step when $k$ is odd.

\medskip  \textbf{If  $k$ is even:} If $k$ is even, then $k-1$ is odd, and there is an H-maximal family on $k-1$ vertices from which we can construct $\binom{k-1}{\left \lfloor \frac{k-1}{2} \right \rfloor}$ Hamiltonian paths. Unlike in the odd case, observe that this is just half the number of Hamiltonian paths we would like to construct, since if $k$ is even then $\binom{k-1}{\left \lfloor \frac{k-1}{2} \right \rfloor}=\frac{1}{2}\binom{k}{\left\lfloor \frac{k}{2} \right\rfloor}$. We start with the exact same transformation of weighted graphs on the vertices $\{x_1, \ldots, x_{k-1}\}$ forming a H-maximal family to weighted graphs on the vertices $\{w,x_1,\ldots, x_{k-1},x'_1,\ldots, x'_{k-1}\}$ forming an MH-maximal family. We have two more vertices: $x_{k}$ and $x'_{k}$ that we did not use yet. Since $k-1$ is odd, the MH-maximal family on $2k-1$ vertices consists of weighted graphs that have a non-empty residual part as every ladder contains an even number of edges from the prescribed matching of the MH-maximal family.  Thus if we add the vertices $x_k,x'_{k}$ connected by an edge of weight $2$ to the shorter path in the residual part, we can complete the residual part of every weighted graph into a ladder, see Figure~\ref{even case}. %This way we doubled the number of Hamiltonian paths that we can get from each weighted graph by the swapping construction, thus we can construct $2\binom{k-1}{\frac{k-2}{2}}=\binom{k}{\frac{k}{2}}$ Hamiltonian paths from our transformed family and the proof is complete.

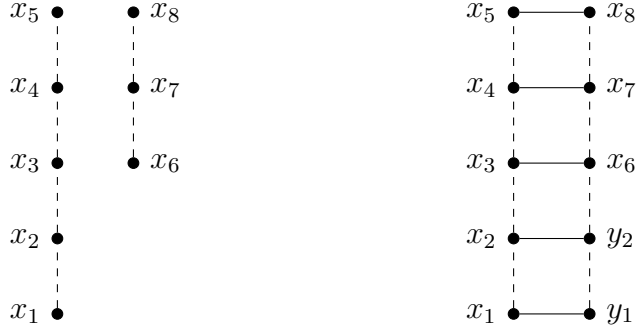
\begin{figure}[htbp,scale=0.5]
\begin{center}
\begin{tikzpicture}
\tikzstyle{vertex}=[draw,circle,fill=black,minimum size=4,inner sep=0]

\node[vertex] (x1) at (0,0) [label=left:$x_1$] {};
\node[vertex] (x2) at (0,1) [label=left:$x_2$] {};
\node[vertex] (x3) at (0,2) [label=left:$x_3$] {};
\node[vertex] (x4) at (0,3) [label=left:$x_4$] {};
\node[vertex] (x5) at (0,4) [label=left:$x_5$] {};
\node[vertex] (x6) at (1,2) [label=right:$x_6$] {};
\node[vertex] (x7) at (1,3) [label=right:$x_7$] {};
\node[vertex] (x8) at (1,4) [label=right:$x_8$] {};
\draw[dashed] (x1) -- (x2);
\draw[dashed] (x2) -- (x3);
\draw[dashed] (x3) -- (x4);
\draw[dashed] (x4) -- (x5);
\draw[dashed] (x6) -- (x7);
\draw[dashed] (x7) -- (x8);

\begin{scope}[shift={(6,0)}]

\node[vertex] (x1) at (0,0) [label=left:$x_1$] {};
\node[vertex] (x2) at (0,1) [label=left:$x_2$] {};
\node[vertex] (x3) at (0,2) [label=left:$x_3$] {};
\node[vertex] (x4) at (0,3) [label=left:$x_4$] {};
\node[vertex] (x5) at (0,4) [label=left:$x_5$] {};
\node[vertex] (x6) at (1,2) [label=right:$x_6$] {};
\node[vertex] (x7) at (1,3) [label=right:$x_7$] {};
\node[vertex] (x8) at (1,4) [label=right:$x_8$] {};
\node[vertex] (y1) at (1,0) [label=right:$y_1$] {};
\node[vertex] (y2) at (1,1) [label=right:$y_2$] {};
\draw[dashed] (x1) -- (x2);
\draw[dashed] (x2) -- (x3);
\draw[dashed] (x3) -- (x4);
\draw[dashed] (x4) -- (x5);
\draw[dashed] (x6) -- (x7);
\draw[dashed] (x7) -- (x8);
\draw[dashed] (y1) -- (y2);
\draw[dashed] (y2) -- (x6);
\draw (x5)--(x8);
\draw(x4)--(x7);
\draw(x3)--(x6);
\draw(x2)--(y2);
\draw(x1)--(y1);

\end{scope}

\end{tikzpicture}
\end{center}
\caption{With the two new vertices $y_1,y_2$, we can complete the residual part to a ladder. }
\label{even case}
\end{figure}
 Thus for each weighted graph we increased the number of ladders by one. This by Lemma~\ref{swapping} doubles the number of Hamiltonian paths that the Z-swapping construction gives. Since for even $k$ we have $2\binom{k-1}{\left \lfloor \frac{k-1}{2} \right \rfloor}=\binom{k}{\left \lfloor \frac{k}{2} \right \rfloor}$ we are done. This finished the case when $k$ is even and the proof is complete.
\end{proof}

\medskip
\par\noindent
{\em Proof of Theorem\ref{main}.}
We have already seen that it is enough to prove the statement when the number of vertices is odd. Then the statement of the theorem is equivalent to say that $H$-maximal families exist on $2n+1$ vertices for every $n$. Lemma~\ref{MH->H} implies that this is true once we know the existence of MH-maximal families for all odd-element vertex sets of size at most $2n+1$. Lemma~\ref{finish} gives that this condition is always satisfied and thus the proof of Theorem~\ref{main} is completed.
\hfill$\Box$

\begin{remark}
Consider the compatibility graph $\mathcal{G}_n$ of the Hamiltonian paths: $V(\mathcal{G}_n)$ is the set of Hamiltonian paths on $n$ vertices and two such vertices are adjacent if the corresponding Hamiltonian paths are compatible. Theorem~\ref{main} determines the clique number $\omega(\mathcal{G}_n)$ of this graph. Observe that the maximal clique is far from unique since in Lemma~\ref{swapping} we can use any ordering of the ladders. Thus we can construct $2^l$ compatible Hamiltonian paths in $l!$ ways there. By using Lemma~\ref{swapping} in the proof of Theorem~\ref{main}, we can construct many cliques of maximal size in $\mathcal{G}_n$. One can actually show that the number of maximal cliques in $\mathcal{G}_n$ is at least doubly exponential.
\end{remark}

\section{Hamiltonian-cycle-different paths} \label{Hcdiff}

Now that we know the maximal number of triangle-different Hamiltonian paths, it is a natural question to ask what happens for other cycles. Observe that since odd cycles are not bipartite, the same upper bound holds for $C_{2k+1}$-different Hamiltonian paths. For small ground sets Table~\ref{final} contains the largest families that we could construct using a computer.

%A computer search for small $n$ revealed what can be seen in Figure~\ref{computer}.

%
%\begin{table}[htbp]
%\begin{center}
%  \begin{tabular}{| c || c | c | c | c | c | c | c |}
%    \hline
%    n  & 3 & 4 &5&6&7&8&9 \\ \hline \hline
%    3-cycle & 3 & 3&10&10&35&35&126 \\ \hline
%    5-cycle &  & &10&10&35&35&126 \\ \hline
%    7-cycle &&&& &21&35&126                   \\ \hline
%    9-cycle  &&&&&&&36\\ \hline
%  \end{tabular}
%\end{center}
%\caption{The maximal size of odd-cycle different Hamiltonian paths that we could construct by a computer.}
%\label{computer}
%\end{table}

Observe that in the case of odd cycles except at the $C_7$-different paths on $7$ vertices and the $C_9$-different paths on $9$ vertices, every value is best possible as they attain the upper bound (the number of balanced bipartitions). By the following Claim the two exceptional values are also best possible. We say that two Hamiltonian paths are Hamiltonian-cycle-different if their union contains a Hamiltonian cycle on their common vertex set. 

\begin{claim} \label{hamcycleupper}
The maximum number of pairwise Hamiltonian-cycle-different Hamiltonian paths on $n$ vertices is at most $\binom{n}{2}.$
\end{claim}
\begin{proof}
We use the trivial fact that a Hamiltonian graph does not contain a vertex of degree one. Let $\mathcal{H}$ be a family of Hamiltonian-cycle-different Hamiltonian paths. For every path in $\mathcal{H}$ we associate the two edges on their ends directed towards the center of the path. Observe that for distinct Hamiltonian paths in $\mathcal{H}$ we associated different edges, as their union does not contain a vertex of degree one. But there are exactly $2\binom{n}{2}$ directed edges in the complete graph on $n$ vertices and we associated two such edges to every path, thus $2|\mathcal{H}| \leq 2 \binom{n}{2}$ finishing the proof.
\end{proof}

\begin{remark}
Note that for a fixed $c$, one can bound the number of $C_{(n-c)}$-different Hamiltonian paths on a ground set of size $n$ in a similar fashion by a polynomial of degree $c+2$.
\end{remark}

We present a construction that attains this bound when the size of the ground set is a prime number.

\begin{claim}
If $p>2$ is a prime then there are $\binom{p}{2}$ pairwise Hamiltonian-cycle-different Hamiltonian paths on $p$ vertices.
\end{claim}
\begin{proof}
Let us draw the ground set on the plane as a regular polygon. Let $\mathcal{H}$ consist of the Hamiltonian paths that do not contain two edges of different length. It is easy to see that $|\mathcal{H}|=\frac{p(p-1)}{2}.$ If both $H_1, H_2 \in \mathcal{H}$  contain edges of length say $l$, then their union consists of every edge of length $l$, and this graph is a Hamiltonian cycle since the size of the ground set is prime. The Hamiltonian path $\{1,2,\ldots, p\}$ is compatible with every other path from $\mathcal{H}$ that uses edges longer than one by the following argument. Let $H_2 \in \mathcal{H}$ be a path that uses edges of length $l$. Since there is a single edge of length $l$ that is missing from $H_2$, either both the edges $((p-l+1),1)$ and $(p-l,p)$ are in $H_2$ or both the edges $(p,l-1)$ and $(1,l)$ are in $H_2$. In both cases we have a Hamiltonian cycle in the union, the latter case can be seen in Figure \ref{hamcycle}.

\begin{figure}[htbp]
\begin{center}
  \begin{tikzpicture}
\tikzstyle{vertex}=[draw,circle,fill=black,minimum size=4,inner sep=0]

\node[vertex] (x1) at (1,0) [label=below:$1$] {};
\node[vertex] (x2) at (1.5,0.1) [label=below:$2$] {};
\node (x3) at (2,0.25)  {};
\node[vertex] (xp) at (0,0) [label=below:$p$] {};
\node[vertex] (xp-1) at (-0.8,0.1) [label=below:$(p-1)$] {};
\node (xp-2) at (-1.3,0.25)  {};

\node (xl-3) at (2.9,1) {};
\node[vertex] (xl-2) at (3,1.5) [label=right:$l-2$] {};
\node[vertex] (xl-1) at (3.1,2) [label=right:$l-1$] {};
\node[vertex] (xl) at (3.15,2.5) [label=right:$l$] {};
\node[vertex] (xl+1) at (3.1,3) [label=right:$l+1$] {};
\node (xl+2) at (2.9,3.5)  {};
\draw[line width=2pt] (x1) -- (x2);
\draw[line width=2pt] (x2) -- (x3);
\draw[line width=2pt] (xl-3) -- (xl-2);
\draw[line width=2pt] (xl-2) -- (xl-1);
\draw (xl-1) -- (xl);
\draw[line width=2pt] (xl) -- (xl+1);
\draw[line width=2pt] (xl+1) -- (xl+2);
\draw[line width=2pt] (xp-1) -- (xp);
\draw[line width=2pt] (xp-1) -- (xp-2);

\draw[line width=2pt] (xp) -- (xl-1);
\draw[line width=2pt] (x1) -- (xl);

    \end{tikzpicture}
\end{center}
\caption{The thick edges form a Hamiltonian cycle. We either have this or the symmetric situation when the two edges of length $l$ are pointing towards the left side.}
\label{hamcycle}
\end{figure}
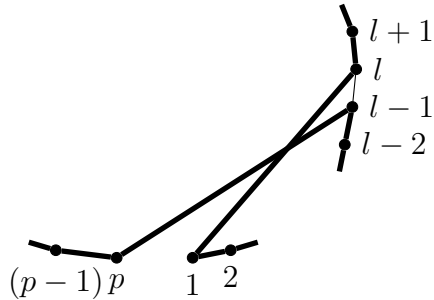

Two arbitrary Hamiltonian paths from $\mathcal{H}$ form a Hamiltonian cycle since we can straighten out one of them by relabelling the vertices so that it becomes the path $\{1,2,\ldots, p\}$, while the other one will be transformed too, but it still connects vertices of the same difference modulo the size of the ground set thus the same reasoning applies.
\end{proof}

\section{Families of graphs with a triangle in every union} \label{other graphs}

In the light of Theorem \ref{intersect}, one might find the following question natural. What is the maximal size of a family of $n$-vertex graphs with the property that every union contains a triangle? This is far easier than the intersection version. Let $H,G$ be graphs. We say that $H$ is a maximal $G$-free graph if it is $G$-free and adding any edge to $H$, the resulting graph has a (not necessarily induced) subgraph isomorphic to $G$. 

\begin{claim} \label{easy}
Let $G$ be a fixed graph. The maximal size of a family of graphs such that every pairwise union contains a copy of $G$ is equal to the size of the family that consists of all the graphs that contain $G$ as a subgraph and all the maximal $G$-free graphs. 
\end{claim}
\begin{proof}
The family that consist of all the graphs that contain $G$ as a subgraph and all the maximal $G$-free graphs clearly has the property that every pairwise union contains a copy of $G$. Now suppose that we have a family $\mathcal{H}$ of graphs with the property that every pairwise union contains a copy of $G$. Suppose that there is a graph $H \in \mathcal{H}$ that does not contain $G$ as a subgraph and that is not maximal $G$-free either. Then $H$ is a subgraph of a maximal $G$-free graph $H'$. Observe that $H' \notin \mathcal{H}$, since otherwise $H \cup H'=H'$ and there is no copy of $G$ in $H'$ a contradiction. Thus we can replace $H$ with $H'$ without losing the property that in each union there is a copy of $G$. Thus one by one we can replace (push up) every graph in $\mathcal{H}$ that does not contain a copy of $G$ and is not maximal $G$-free to be maximal $G$-free. This finishes the proof.  
\end{proof}

Proposition \ref{easy} is relevant to us when $G$ is a triangle.  Thus we see that when we do not add any restriction, the optimal families with the property that any union contains a triangle  consist of the graphs that contain a triangle (the trivial ones) and the maximal triangle-free graphs. A less general superclass of the class of Hamiltonian paths is the class of trees. But determining the maximal size of triangle-different families of trees is also simple.

\begin{claim} \label{recur}
The maximum number of pairwise triangle-different trees on $n$ vertices is $2^{n-1}-1.$
\end{claim}
\begin{proof}
For the upper bound, observe that every tree is a bipartite graph. The number of (not necessarily balanced) bipartitions of $[n]$ is $2^{n-1}$. We clearly cannot have two trees with the same bipartition, as the union would be bipartite. Moreover no tree corresponds to the bipartition of the ground set where one side is the empty set. 

For the lower bound, we construct a large enough family of trees inductively. For $n=2$ the only tree on two vertices attains the upper bound. Suppose that we have a family $\mathcal{F}_{n-1}$ of $2^{n-2}-1$ triangle-different trees on the vertices $\{v_1, \ldots, v_{n-1}\}$. We build a new set of trees $F_{n}$ on the vertices $\{v_1, \ldots, v_{n}\}$ as follows.  For each tree $T \in \mathcal{F}_{n-1}$ let us fix two adjacent vertices $x=x(T)$ and $y=y(T)$ and we build two new trees $T_{x}$ and $T_{y}$ by connecting the new vertex $v_{n}$ as a leaf to $x$ and $y$, respectively, without changing anything else in the tree. We also add the star centered at $v_n$ to $\mathcal{F}_n$. Clearly every other tree is triangle-different from the star centered at $v_n$. Two trees in $\mathcal{F}_n$ that are constructed from different trees in $\mathcal{F}_{n-1}$ are triangle-different. And for a fixed $T \in \mathcal{F}_{n-1}$, the union of the trees $T_{x}$ and $T_{y}$ contains the triangle $x,y,v_n$. Thus $\mathcal{F}_n$ is a triangle-different family, with size $|\mathcal{F}_{n}|=2|\mathcal{F}_{n-1}|+1=2^{n-1}-1$.
\end{proof}

\begin{remark}
One can also ensure that the trees constructed in the proof of Claim \ref{recur} are either stars or double stars (two vertex disjoint stars and an edge connecting the center of the two stars) by the choice of the vertices $x$ and $y$.
\end{remark}

\section{Open problems and concluding remarks} \label{opr}

We conjecture that 
%something like 
a relaxed version of 
Theorem~\ref{main} should be true for any odd cycle.

\begin{conj}\label{eventually}
If  $k > 1$ and $n$ is large enough, the maximal number of $C_{2k+1}$-different Hamiltonian paths on $n$ vertices is equal to the number of balanced bipartitions of the ground set $n$.
\end{conj}

On the other hand the maximal number of $C_{2k}$-different Hamiltonian paths is more than exponential (One can construct such a system of size $\left( \frac{n}{k}\right)!$ using the simple method in Theorem~1 of \cite{k=4}). An other significant difference between the even and the odd case is that the maximal number of even-cycle-different Hamiltonian paths is asymptotically larger than the maximal number of $C_4$-different Hamiltonian paths, as expected, see \cite{k=4}. Thus one is tempted to think that we can observe the "normal behaviour" for even cycles. For the following weaker version of Conjecture~\ref{eventually} we have additional supporting evidence.

\begin{conj} \label{2}
If  $k > 1$ then the maximal number of $C_{2k+1}$-different Hamiltonian paths on $n$ vertices is at least $2^{n-o(n)}$.
\end{conj}
Generalizing the methods of present paper, we managed to show that Conjecture~\ref{2} is true when $2k$ is a power of two. But even in these cases our constructions are only asymptotically equal to the upper bound of balanced bipartitions. We consider this as the basis of a subsequent paper.

\medskip

We also conjecture that the bound in Claim~\ref{hamcycleupper} can also be attained for composite integers.

\begin{conj}\label{hamconj}
The maximum number of pairwise Hamiltonian-cycle-different Hamiltonian paths is $\binom{n}{2}$ for every $n$.
\end{conj}

The maximal number of $C_4$- or $K_4$-different Hamiltonian paths is also an open question at the time, even the correct order of magnitude is unknown. For other open problems about the maximal number of Hamiltonian paths with some restrictions on the pairwise unions see: \cites{connector,locsep,degree4,k=4,komesi,indep}.
Finally, in Table~\ref{final} we present the sizes of the largest families that we could construct for small ground sets and cycle lengths. This provides some experimental evidence for Conjecture~\ref{eventually} and Conjecture~\ref{hamconj}.

\begin{table}[htbp]
\begin{center}
\renewcommand*{\arraystretch}{1.2}
  \begin{tabular}{| c || c | c | c | c | c | c | c |}
    \hline
    n  & 3 & 4 &5&6&7&8&9 \\ \hline \hline
    3-cycle & 3 & 3&10&10&35&35&126 \\ \hline
    4-cycle && 6&12&32&97&248&594 \\ \hline
    5-cycle &  & &10&10&35&35&126 \\ \hline
    6-cycle &&&&15&49&128&315 \\ \hline
    7-cycle &&&& &21&35&126                   \\ \hline
    8-cycle &&&&&&28& 135 \\ \hline
    9-cycle  &&&&&&&36\\ \hline
  \end{tabular}
\end{center}
\caption{Lower bounds to the maximal size of cycle-different Hamiltonian paths.}
\label{final}
\end{table}

\section{Acknowledgement}

We would like to thank Kitti Varga and Géza Tóth for the beautiful proofs of the combinatorial identity in Lemma \ref{MH->H}. We would also like to thank Gábor Simonyi, Géza Tóth and János Körner for their help which greatly improved the quality of this manuscript. Finally we also thank many of our colleagues for encouragement and their interest in our progress.

\end{document}